\documentclass{amsart}  
\usepackage{amsmath}
\usepackage{amsfonts,amssymb}
\usepackage{tikz}
\usepackage{pgfplots}
\pgfplotsset{compat=1.10}

\pgfplotsset{grid style={dotted, gray!80!white}}

\usepackage{graphicx,xcolor}
\usepackage{bbm,enumerate}
\usepackage{esint}
\usepackage{colortbl}

\numberwithin{equation}{section}\newtheorem{theorem}{Theorem}[section]
\newtheorem{corollary}[theorem]{Corollary}\newtheorem{lemma}[theorem]{Lemma}

\theoremstyle{remark}

\theoremstyle{definition}

\newcommand{\supp}{\mathrm{supp\,}}
\newcommand{\R}{\mathbb{R}}

\usepackage{mathrsfs}

\newcommand{\HH}{\mathcal{H}}
\newcommand{\QQ}{\mathcal{Q}}

\begin{document}

\title{{A note on pointwise convergence for the Schr\"odinger equation}}

\author{Renato Luc\`a\and Keith M. Rogers}

\address{Departement Matematik und Informatik, Universit\"at Basel, 4051, Switzerland.
}
\email{renato.luca@unibasel.ch}

\address{Instituto de Ciencias Matem\'aticas CSIC-UAM-UC3M-UCM, Madrid, 28049, Spain.
}
\email{keith.rogers@icmat.es}

\thanks{Supported by the ERC grants 277778 and 676675, the MINECO grants SEV-2015-0554 and MTM2013-41780-P (Spain), and the NSF grant DMS-1440140 (MSRI, Berkeley, Spring 2017).
 }

\maketitle
\begin{abstract}
We consider Carleson's problem regarding pointwise convergence for the Schr\"o\-dinger equation. Bourgain recently 
proved that there is initial data, in $H^s(\R^n)$ with $s<\frac{n}{2(n+1)}$, for which the solution diverges on a set of 
nonzero Lebesgue measure. We provide a different example enabling the generalisation to fractional Hausdorff measure.
\end{abstract}

\section{Introduction}

Consider the  Schr\"odinger equation, $i \partial_{t} u +\Delta u=0$, in $\R^{n+1}$, with 
initial data $u(\cdot, 0) = u_0$ in the Bessel potential/Sobolev space defined by
$$
H^{s}(\R^{n}) = (1-\Delta)^{-s/2}L^2(\R^{n}):=\left\{ 
G_{s} * g \, : \,  g \in L^{2}(\R^n)
\right\} \, .
$$
The Bessel kernel $G_{s}$ is defined as usual via its Fourier transform; $\widehat{G}_{s} = (1 + |\cdot|^{2})^{- s/2}$.
In~\cite{Carl},  Carleson proposed the problem of identifying 
the exponents $s > 0$ for which
\begin{equation}\label{CarlesonProblem}
\lim_{t\to 0} u(x , t) = u_0(x), \qquad \text{a.e.} \quad x \in \R^n, \qquad \forall \  u_0\in H^s(\R^{n}) \, ,
\end{equation}
with respect to Lebesgue measure, and proved that \eqref{CarlesonProblem} holds as long as
 $s \geq 1/4$ and $n=1$. Dahlberg and Kenig  then showed that this condition is necessary,  providing a complete solution in the one-dimensional case~\cite{DahlKenig}.

 In higher dimensions, \eqref{CarlesonProblem} holds as long as
 $s > \frac{2n-1}{4n}$; see \cite{L, B}. 
 It was thought that $s\ge 1/4$ might also be sufficient in higher dimensions (see for example \cite{GS} or \cite{T}), however 
 Bourgain recently proved that~$s\ge\frac{n}{2(n+1)}$ is necessary \cite{Bnew}. Since then, Du, Guth and Li \cite{DGL} improved the sufficient condition in two dimensions to the almost sharp $s>1/3$. 
 
 Here we give a new proof of the necessary condition using a different example (fewer frequencies travelling in a skew direction; see \eqref{fg}). 
 We replace number theoretic arguments, via comparison with Guass sums, with ergodic arguments that exploit 
 the occasional complete absence of cancelation as in \cite{LuR2}. 
This permits us to generalise to fractional Hausdorff measure.
When $n=2$, the proof becomes much simpler as the ergodic arguments are trivial in that case.

\begin{theorem}\label{Thm:DirectConv}
Let $ (3n+1)/4 \leq  \alpha \leq n$.  
Then, for any 
\begin{equation}\label{INTVAL}
s < \frac{n}{2(n+1)}+\frac{n-1}{2(n+1)}(n-\alpha) \, ,
\end{equation}
there exists $u_0 \in H^{s}(\mathbb{R}^{n})$ such that 
$$
\limsup_{t \to 0} | u(x,t) | = \infty
$$ 
for all $x$ in a set 
of positive $\alpha$--Hausdorff measure. 
\end{theorem}

The study of this refined version of Carleson's problem was initiated by Sj\"ogren and Sj\"olin \cite{SS}. Theorem~\ref{Thm:DirectConv} improves \cite[Theorem 2]{LuR2}, although the result there holds for the full range $n/2\le \alpha\le n$ (with $\alpha<n/2$ the question was previously resolved in \cite{BBCR}).
It has been conjectured that $s\ge\frac{n}{2(n+1)}$ should also be sufficient in the $\alpha=n$ case; see \cite{DG}. If that were true, then \eqref{INTVAL} would represent the interpolating condition between two sharp results, and so it would be interesting to see if Theorem~\ref{Thm:DirectConv} could be extended to the range $n/2\le \alpha\le n$, or whether there is a discontinuity in behaviour as in the one-dimensional case. 

Indeed, defining
$$\alpha_n(s):=\sup_{u_0\in H^s(\R^n)}\mathrm{dim}\Big\{\ x\in\mathbb{R}^n\ :\ \limsup_{t\to0}|u(x,t)|=\infty \ \Big\} \, ,$$ 
where $\mathrm{dim}$ denotes the Hausdorff dimension, the combination of Theorem~\ref{Thm:DirectConv} with previous results yields 
\begin{equation*}
\alpha_n(s)\ge \left \{
\begin{array}{rcccccccl}
  &n   &\text{when}&\!\!\!\!\!& \!\!&s\!\!&< &\!\!\!\frac{n}{2(n+1)}&\\ [0.8ex]
 & n+\frac{n}{n-1}-\frac{2(n+1)s}{n-1}   &\text{when}& \frac{n}{2(n+1)}\!\!\!\!\!&\le\!\!& s\!\!&<&\!\!\!\frac{n+1}{8}&\\ [0.8ex]
  &n+1-\frac{2(n+2)s}{n} &\text{when}&\frac{n+1}{8}\!\!\!\!\!& \le\!\!& s\!\!&<& \!\!\!\frac{n}{4}&\\  [0.8ex]
 & n-2s\qquad &\text{when}&\frac{n}{4} \!\!\!\!\!& \le \!\!& s\!\!&\le& \!\!\!\frac{n}{2}&\!\!\!\!\!\!\!\! \, . 
\end{array}\right.
\end{equation*}
The function on the right-hand side is continuous apart from a jump of $\frac{1}{2n}$ over the regularity $s=\frac{n+1}{8}$. The bound is best possible in one dimension, in which case the central intervals are empty and the dimension jumps by a half over $s=1/4$. This is a consequence of the Dahlberg--Kenig example combined with~\cite{BBCR}, where it was proven that $\alpha_n(s)\le n-2s$ in the range $n/4\le s\le n/2$. For the best known upper bounds with lower regularity,  see \cite[Theorem 1.2]{LuR}.

In the following section we present the quantitive ergodic lemma that will be used in the third section to provide a new proof that  $s\ge\frac{n}{2(n+1)}$ is necessary in the Lebesgue measure case.  For this we will employ the Niki\v sin--Stein maximal principle. However, in the fourth section, we will explicitly construct  data for which the divergence occurs, see \eqref{TestinfFunctBennFin}, enabling the proof of Theorem~\ref{Thm:DirectConv}.

\section{A quantitive ergodic lemma}

It is well-known that linear flow on the torus, in most directions, eventually passes arbitrarily close to every point. This remains true 
when only considering equidistant points on the trajectory.

\begin{lemma}\label{erg}
Let $d \geq 2$, $0<\varepsilon,\delta<1$ and $\kappa>\frac{1}{d+1}$. Then, if $\delta<\kappa$ and $R>1$ is sufficiently large, 
there is  $\theta \in \mathbb{S}^{d-1}$ for which, given any $y\in \mathbb{T}^{d}$  and $a\in\mathbb{R}$, there is a $t_y\in R^\delta\mathbb{Z}\cap(a,a+R)$ such that
\begin{equation*}
|y-t_y\theta|\le \varepsilon R^{(\kappa-1)/d} \, .
\end{equation*}
Moreover, this remains true with $d=1$, for some $\theta\in(0,1)$.
\end{lemma}

\begin{proof} When $d=1$, by taking $\theta$ close to $R^{-1}$, we obtain approximately~$R^{1-\delta}$ points~$t_y\theta$ equally spaced at intervals of length $R^{\delta-1}$ on the circle. For each $y\in \mathbb{T}$, one of these points $t_y\theta$ must lie closer than a distance of $\varepsilon R^{\kappa-1}$ if $R$ is sufficiently large so that $R^{\delta-\kappa}<\varepsilon$.

When $d\ge 2$ and $a=0$, this was proved in~\cite[Lemma~2]{LuR2}. The adjustment to the general case $a\in\R$ amounts to little more than starting the flow at different points on the translation invariant torus. One can also easily check that the proof in \cite{LuR2} is essentially unchanged. One need only translate their function $\eta_R$ by $a$, and the modulus of the Fourier transform of this is unchanged, so the remainder of the argument is exactly the same.
\end{proof}

The following corollary is optimal, in the sense that the statement fails for larger~$\sigma$. 
To see this, we can place balls 
of radius~$\varepsilon R^{-\frac{\gamma}{d}}$ centred at the points of the sets below and assume that the balls are disjoint. 
Then the volume of such a set would be of the order $\varepsilon^dR^{d+1/2-\gamma-(d+2)\sigma}$, a quantity that 
is arbitrarily small for larger $\sigma$. 
 Neither is it possible to extend the range of~$\gamma$, as then the set of times~$t$ could be empty. To avoid this we must 
have $\sigma < 1/4$ which is ensured by the restriction~$\gamma \geq 3d/4$.

\begin{corollary}\label{Corollary:ToroImpr}
Let $d\ge 2$, $\frac{3d}{4}\le \gamma \le d$ and $0< \sigma < \frac{1+2(d- \gamma)}{2(d+2)}$. 
Then, for any $\varepsilon>0$ and sufficiently large $R >1$, 
there exists  $\theta \in \mathbb{S}^{d-1}$ such that
\begin{equation*}
\bigcup_{t\in R^{2\sigma-1}\mathbb{Z}\cap(a,a+R^{-1/2})} \big\{x\in R^{\sigma-1} \mathbb{Z}^{d}\, :\, |x|\le 2\big\}+t\theta
\end{equation*}
is $\varepsilon R^{-\frac{\gamma}{d}}$-dense in $B(0,1/2)$, for all $a\in (0,1)$. Moreover, this remains true with $d=1$, for some $\theta\in(0,1)$.
\end{corollary}

\begin{proof}
We first rescale by $R^{1-\sigma}$,
and then replace $R^{1/2-\sigma}$ by $R$. In this way the statement is equivalent to proving that, 
for any $y\in B(0,R^{\frac{1-\sigma}{1/2-\sigma}}/2)$ there exists 
$$
x_{y} \in \mathbb{Z}^{d}\cap B(0,2R^\frac{1-\sigma}{1/2-\sigma})\quad \text{and}\quad 
t_{y}\in R^{\frac{\sigma}{1/2-\sigma}} \mathbb{Z}\cap (R^{\frac{1-\sigma}{1/2-\sigma}}a,R^{\frac{1-\sigma}{1/2-\sigma}}a+R)
$$ 
such that
\begin{equation*}\label{EquivTorusImpr}
|y - (x_{y} + t_{y}\theta) |\ < \ \varepsilon R^{\frac{1-\sigma}{1/2-\sigma}-\frac{\gamma}{d(1/2-\sigma)}} \, ,
\end{equation*}
for a fixed $\theta \in \mathbb{S}^{d-1}$, independent of $y$ and $a$. 
By taking the quotient $\R^n / \mathbb{Z}^{d} = \mathbb{T}^{d}$, this would follow if, for any $[y] \in \mathbb{T}^{d}$, we have 
\begin{equation*}
|[y] - [t_{y} \theta]| \ < \ \varepsilon R^{\frac{1-\sigma}{1/2-\sigma}-\frac{\gamma}{d(1/2-\sigma)}} \, .
\end{equation*} 
Now this is a consequence of Lemma~\ref{erg}, by taking $\delta=\sigma/(1/2-\sigma)$ and $\kappa$ so that 
$$ 
\frac{\kappa-1}{d}=\frac{1-\sigma}{1/2-\sigma}-\frac{\gamma}{d(1/2-\sigma)} \, .
$$
The conditions $0<\delta<1$ and $\delta < \kappa$ are then ensured by the restrictions on $\gamma$ and
$\sigma$ in the statement. 
\end{proof}

\section{Proof of the Lebesgue measure necessary condition}

When  the initial data $u_{0}$ is a Schwartz function, the solution $u$ to the Schr\"odinger equation can be written as
\begin{equation*}
u(x, t) = e^{i t \Delta}u_{0}(x):=\frac{1}{(2\pi)^{n/2}}\int_{\R^n}\widehat{u}_{0}(\xi)\, e^{ix\cdot\xi -it | \xi |^{2}} d \xi \, . 
\end{equation*}
By the Niki\v sin--Stein maximal
principle \cite{N, st}, it suffices to prove the  following theorem.

\begin{theorem}\label{max2}
Suppose that there is a constant $C_s$ such that  
\begin{equation*}
\left\| \sup_{0< t < 1} \big| e^{it\Delta} f \big| \right\|_{L^{2}(B(0,1))} \le C_s\| f \|_{H^{s}(\R^n)} \, , 
\end{equation*}
whenever $f$ is a Schwartz function. Then $s\ge\frac{n}{2(n+1)}$.
\end{theorem}

\begin{proof}
Writing $t/(2\pi R)$ in place of $t$, the maximal estimate implies that\footnote{We write $a\lesssim b$ ($a\gtrsim b$) whenever $a$ and $b$ are nonnegative quantities that satisfy $a \leq C b$ ($a \geq C b$) for a  constant $C > 0$. We write $a\simeq b$ when $a\lesssim b$ and $b\lesssim a$.}
\begin{equation}\label{otooBis0}
\left\| \sup_{0<t<1} \big| e^{i\frac{t}{2\pi R}\Delta} f \big| \right\|_{L^{2}(B(0,1))} \lesssim R^s\| f \|_{2} \, ,
\end{equation}
whenever $\supp \widehat{f} \subset B(0,2R)$ and $R>4$. From now on we let $B(0,\rho)$ denote the $(n-1)$-dimensional ball of radius $\rho>0$, a fixed, sufficiently small constant.
Writing $x=(x_1,\bar{x})$ and letting $0 < \sigma < \frac{1}{2(n+1)}$, we consider frequencies in the set 
\begin{equation}\nonumber
\Omega := \big\{ \bar{\xi}\in 2\pi R^{1-\sigma} \mathbb{Z}^{n-1}  \,:\, |\bar{\xi}|\le R \big\} + B(0,\rho) \,,
\end{equation}
and Schwartz functions defined by $\widehat{\phi}=\chi_{(-\rho,\rho)}$  and 
$\widehat{g} =\chi_{\Omega}$. Then the initial data is defined by
\begin{equation}\label{fg}
f(x)=   e^{i\pi R(1,\theta)\cdot x}\phi(R^{1/2}x_1)g(\bar{x}),  
\end{equation}
where $\theta \in (0,1)$ when $n=2$ and $\theta \in \mathbb{S}^{n-2}$ in higher dimensions.

 Note that the solution factorises
\begin{equation}\label{eq:factrization}
e^{i\frac{t}{2\pi R}\Delta} f(x) =e^{i\frac{t}{2\pi R}\Delta} f_{dk}(x_1) e^{i\frac{t}{2\pi R}\Delta} f_{\theta}(\bar{x}) \, , 
\end{equation}
where  $f_{dk}$ and $f_\theta$ are defined  by
$$
f_{dk}(x_1)=e^{i\pi Rx_1}\phi(R^{1/2}x_1)\quad\text{and}\quad f_\theta(\bar{x}) = e^{i\pi R\theta\cdot \bar{x}} g(\bar{x}) \, .
$$
By a change of variables, we have
\begin{align}\label{dk}
|e^{i\frac{t}{2\pi R}\Delta} f_{dk}(x_1)|&=\frac{R^{-1/2}}{(2\pi)^{1/2}}\Big|\int \widehat{\phi}\big(R^{-1/2}(\xi_1-\pi R)\big)e^{ ix_1\xi_1 -i\frac{t}{2\pi R} \xi_1^{2}} d \xi_1\Big|\\\nonumber
&=\frac{1}{(2\pi)^{1/2}}\Big|\int_{-\rho}^\rho e^{ iR^{1/2}(x_1-t)y-i\frac{t}{2\pi} y^{2}} d y\Big|\simeq 1 \, ,
\end{align}
whenever $|t|\le 1$ and $|x_1-t|\le R^{-1/2}$.  
Indeed, these restrictions ensure that the phase is close to zero, so that no cancelation occurs in the integral. 
By Plancherel's identity and Fubini's theorem,
\begin{equation}\nonumber
\|f\|_2=\|f_{dk}\|_2\|f_\theta\|_2\simeq R^{-1/4}|\Omega|^{1/2}\, ,
\end{equation} 
so that plugging the data into the maximal estimate \eqref{otooBis0} and using \eqref{eq:factrization} and \eqref{dk}, 
we obtain
\begin{equation}\label{otooBis}
\Big(\int_{B(0,1)} \int_0^{1/2} \sup_{t\in(x_1,x_1+R^{-1/2})} \big| e^{i\frac{t}{2\pi R}\Delta} f_\theta(\bar{x}) \big|^2dx_1d\bar{x}\Big)^{1/2}  \lesssim R^sR^{-1/4}|\Omega|^{1/2} \, .
\end{equation}

In order to understand the behaviour of $e^{i\frac{t}{2\pi R}\Delta} f_\theta$ we first consider the unmodulated version $e^{i\frac{t}{2\pi R}\Delta} g$.
Barcel\'o, Bennett, Carbery, Ruiz and Vilela \cite{BBCRV} showed  that
\begin{equation}\label{Phase=1}
|e^{i\frac{t}{2\pi R}\Delta} g(\bar{x})| \gtrsim |\Omega|,
\quad
\mbox{for all}
\quad
(\bar{x},t) \in X_0\times R^{2\sigma-1}\mathbb{Z}\cap(0,1) \, ,
\end{equation}
where, with $\varepsilon$ sufficiently small, $X_0$ is defined by
\begin{equation}\nonumber
X_0 = \big\{ \bar{x}\in R^{\sigma-1}\mathbb{Z}^{n-1}\,:\, |\bar{x}|\le 2\big\}+ B(0,\varepsilon R^{-1}) \, .
\end{equation}
This time the phase in the integrand never strays too far from zero modulo~$2\pi i$, and so again there is no cancelation in the integral. Now
\begin{equation}\nonumber
 (\bar{x},t)\in X_{t\theta} \times R^{2\sigma - 1} \mathbb{Z}\cap(0,1) \quad 
 \Rightarrow\quad (\bar{x}-t \theta,t) \in X_0\times R^{2\sigma-1}\mathbb{Z}\cap(0,1) \, ,
\end{equation}
where $X_{t\theta}:=X_0+t\theta$ and $
 \big| e^{i\frac{t}{2\pi R}\Delta} f_\theta(\bar{x})|= \big| e^{i\frac{t}{2\pi R}\Delta} g(\bar{x} -t\theta )|
$.
Combining this fact with (\ref{Phase=1}) yields
\begin{equation*}\label{otooBisr}
\sup_{t\in(x_1,x_1+R^{-1/2})} \big| e^{i\frac{t}{2\pi R}\Delta} f_\theta(\bar{x})|  \gtrsim |\Omega|,
\quad
\text{for all}\quad
\bar{x} \in \Gamma_{\!x_1} :=\!\!\!\!\bigcup_{t\in R^{2\sigma - 1} \mathbb{Z}\cap(x_1,x_1+ R^{-1/2})}\!\!\!\!X_{t\theta} \,, 
\end{equation*} and this holds uniformly for all $x_1\in (0,1/2)$. 

Now the sets  $\Gamma_{\!x_1}$ can be considered to be $\varepsilon R^{-1}$--neighbourhoods of the sets of Corollary~\ref{Corollary:ToroImpr}. So, taking $\gamma=d=n-1$, 
there is a~$\theta\in \mathbb{S}^{n-2}$  so that  $B(0,1/2)\subset \Gamma_{\!x_1}$  for 
all $x_1\in(0,1/2)$. 
Substituting  into \eqref{otooBis}, this yields
$$
|\Omega|^{1/2} \lesssim R^sR^{-1/4} \, . 
$$
As $|\Omega|\simeq R^{(n-1)\sigma}$, we can let $\sigma$ tend to $\frac{1}{2(n+1)}$ and then $R$ tend to infinity, so that
$$
s\ge 1/4+\frac{n-1}{4(n+1)}=\frac{n}{2(n+1)} \, ,
$$
which completes the proof.
\end{proof}

\section{Proof of Theorem \ref{Thm:DirectConv}}

The solution is typically represented as $u(\cdot,t):= \lim_{N \to \infty} S_{N}(t)u_0$, where
\begin{equation}\label{rt}
S_{N}(t)u_0(x) :=
\frac{1}{(2\pi)^{n/2}} \int_{\mathbb{R}^n}
\Psi(N^{-1}\xi)\,
\widehat{u}_0(\xi)\, e^{ ix\cdot\xi -  i t | \xi |^{2} } d \xi \, ,
\end{equation}
and $\Psi$ is a fixed function, equal to one near the origin, that decays in such a way that the integral is well-defined. For
 convenience we take $\Psi(\xi)=\prod_{j=1}^{n} \psi(\xi_{j})$, where~$\psi$ is differentiable, supported in the interval $[-2,2]$ and equal to one on $[-1,1]$. The limit is usually taken with respect to the $L^2$--norm, but here we will take all limits pointwise, at each point that they exist. Supposing that $\alpha>n-2s$, as we may, the limits exist at almost every $x$ with respect to $\alpha$--Hausdorff measure (see for example \cite[Corollary 17.6]{M}) and they coincide with the usual $L^2$--limit almost everywhere with respect to Lebesgue measure.

We take $0<  \sigma < \frac{1+2(n-\alpha) }{2(n+1)}$
and $\lambda := 2^{\frac{M}{1-\sigma}}$, with $M\in \mathbb{Z}$ to be chosen sufficiently large later. 
As $\alpha \geq (3n+1)/4$ we have that $\sigma < 1/4$. 
Writing $\xi=(\xi_1,\bar{\xi})$,  we consider the sets of frequencies 
\begin{equation}\nonumber
\Omega^{j} = \big\{ \bar{\xi} \in 2 \pi \lambda^{j(1-\sigma)}  \mathbb{Z}^{n-1}  \,:\, \lambda^{j} \leq | \xi_{m} |< \lambda^{j+1}, \,  m=2,\ldots, n \big\} + Q\Big( 0,\frac{\varepsilon_{1}}{\sqrt{n\!-\!1}} \Big) \, ,
\end{equation}
where
$\varepsilon_{1} > 0$ is a fixed sufficiently small constant and $j\in \mathbb{N}$. 
Here~$Q(0,\ell)$ 
is the closed~$(n-1)$-dimensional cube centred at the origin with side-length  $\ell$, and we denote its interior by $\mathring{Q}(0,\ell)$.
For a suitable choice of $\theta_{j} \in (0,1)$ when $n=2$ or $\theta_{j}\in \mathbb{S}^{n-2}$ in higher dimensions, the initial data  $u_0$ that gives rise to a divergent solution is given by
\begin{equation}\label{TestinfFunctBennFin}
u_0(x) := \sum_{j \in \mathbb{N}} e^{i\pi \lambda^j(1,\theta_{j})\cdot x}\phi(\lambda^{j/2}x_1)g_{j}(\bar{x}). 
\end{equation}
Here $\widehat{\phi}=\chi_{(-\varepsilon_{1},\varepsilon_{1})}$ and
$\widehat{g}_{j} :=  \lambda^{j \delta } |\Omega^j|^{-1}\chi_{\Omega^{j}}$, with $0< \delta< \sigma/4$.
Noting that $| \Omega^{j}| \simeq \lambda^{j (n -1)\sigma}$, we have that $u_0\in H^{s}$ whenever
\begin{equation*}\label{URONS}
s < \frac{(n-1) \sigma }{2}+\frac{1}{4}  -  \delta \, .
\end{equation*}
Eventually we will let $\sigma$ tend to $ \frac{1+2(n-\alpha) }{2(n+1)}$ and $\delta$ tend to zero, covering all the cases  of the range \eqref{INTVAL}.

First we consider $(n-1)$-dimensional data given by $f_{\theta_{j}}(\bar{x}) := e^{ i\pi \lambda^{j} \theta_{j} \cdot \bar{x} }g_{j}(\bar{x})$ and the associated solutions on $X_{t\theta_{j}}^{j} \times 
T^{j}_{x_1}$ defined by
\begin{equation*}
X_{t\theta_j}^{j}= \{ \bar{x} \in \lambda^{j(\sigma-1)}\mathbb{Z}^{n-1} \, : \, |\bar{x}| \leq 2 \} + \mathring{Q} ( t\theta_{j}, \varepsilon_{2} \lambda^{-j} ) \, ,
\end{equation*}
$$
T^{j}_{x_1} = \big\{ t\in  \lambda^{j(2\sigma-1)}\mathbb{Z}\, : \, x_1<t<x_1+\lambda^{-j/2}\big\} \, .
$$
Taking $\varepsilon_{1}$, $\varepsilon_{2}$ sufficiently small and $x_1\in(0,1/2)$, as in the previous section we have 
\begin{equation}\label{GalInvPreq}
\Big| S_{N}  \Big(  \frac{t}{2\pi \lambda^{j}}  \Big)  f_{\theta_{j}}(\bar{x}) \Big|  \gtrsim \lambda^{j \delta},
\quad
\mbox{for all}
\quad
(\bar{x}, t) \in X^{j}_{t\theta_{j}} \times T^{j}_{x_1} \, 
\end{equation}
whenever $N\ge 2\pi \lambda^{2j}$; see \cite[eq. 18]{LuR2}.
On the other hand, in \cite[eq. 20]{LuR2} it was proven that for~$k > 2j$, we have  \begin{equation}\label{GIPBIS2}
\Big| S_{N}  \Big(  \frac{t}{2\pi \lambda^{j}} \Big)  f_{\theta_{k}}(\bar{x}) \Big|  
\lesssim
\lambda^{-k \delta},
\quad
\mbox{for all}
\quad
(\bar{x}, t) \in  \R^{n-1} \times T^{j}_{x_1} \, ,
\end{equation}
whenever $N\ge 2\pi \lambda^{2j}$. It is something of a nuisance that this does not quite hold for all $k>j$. To circumvent this, we consider  
\begin{equation}\nonumber
X^{k, \delta}_{\lambda^{k-j} t\theta_{k}} :=\lambda^{k(\sigma-1)} \mathbb{Z}^{n-1} +  Q(\lambda^{k-j} t \theta_{k},\varepsilon_{2} \lambda^{-k(1 - 2\delta)}) \, .
\end{equation}
In \cite[eq. 19]{LuR2} it was proven that, when $j<k\le 2j$ and $x_1\in(0,1/2)$,
\begin{equation}\label{GIPBIS}
\Big| S_{N}  \Big(  \frac{t}{2\pi \lambda^{j}} \Big)  f_{\theta_{k}}(\bar{x}) \Big|  
\lesssim
\lambda^{-k \delta},
\quad
\mbox{for all}
\quad
(\bar{x}, t) \in ( \R^{n-1} \setminus X^{k, \delta}_{\lambda^{k-j} t\theta_{k}} ) \times T^{j}_{x_1} \, 
\end{equation}
whenever $N\ge 2\pi \lambda^{2j}$.
Thus, considering  
\begin{equation*}
\Gamma_{\!t\theta_{j}}^{j} := X_{t\theta_{j}}^{j} \setminus  \bigcup_{j<k \le 2j } X^{k, \delta}_{\lambda^{k-j} t\theta_{k}}\quad \mbox{and}\quad
\Gamma^{j}_{\!x_1} := \bigcup_{ t \in T^{j}_{x_1} } \Gamma_{\!t \theta_{j}}^{j} \, ,
\end{equation*}
an immediate consequence of~\eqref{GalInvPreq}, \eqref{GIPBIS} and~\eqref{GIPBIS2} is that if~$x \in \Gamma^j$, defined by
\begin{equation*}
\Gamma^j=\left\{x\in\R^n \, : \,  x_1\in(0,1/2),\quad \bar{x}\in\Gamma^{j}_{\!x_1}\right\} \, ,
\end{equation*}
there exists a time $t_{j}(x) \in T^{j}_{x_1}$ such that
\begin{equation}\nonumber
\mbox{(i)} \
     \Big| S_{N} \Big( \frac{t_{j}(x)}{2\pi \lambda^{j}}  \Big)  f_{\theta_{j}}(\bar{x}) \Big|  \gtrsim \lambda^{j \delta };
\qquad
\mbox{(ii)}  \
     \Big| S_{N} \Big( \frac{t_{j}(x)}{2\pi \lambda^{j}} \Big)  f_{\theta_{k}}(\bar{x}) \Big|  
     \lesssim   \lambda^{-k \delta} 
     \quad
     \mbox{for all}
     \quad
     k > j \, .
\end{equation}

Now divergence occurs on the set of $x$ that belong to infinitely many $\Gamma^{j}$; that is \begin{equation*}
 \Gamma := \bigcap_{j\ge1}   \bigcup_{k\ge j}  \Gamma^{k} \, .
\end{equation*}
To see this, we note that if $x \in  \Gamma$ there exists an infinite subset $J(x)\subset\mathbb{N}$ with an associated 
sequence of times
$t_{j}(x) \in T^{j}_{x_1}$, for all $j\in J(x)$, such that both (i) and (ii) are satisfied.
The solution factorises as in~\eqref{eq:factrization}, so that, recalling~\eqref{dk}, 
we see that the properties (i) and (ii) remain true while considering the extension $f_j$, defined by
$$
f_j(x)=e^{i\pi \lambda^jx_1}\phi(\lambda^{j/2}x_1)f_{\theta_{j}}(\bar{x}) \, .
$$
Now, since $u_0=\sum_{j\ge 1} f_j$, by the triangle inequality
\begin{equation}\nonumber
\Big| S_{N} \Big( \frac{t_{j}(x)}{2\pi \lambda^{j}} \Big)  u_0(x) \Big|
\gtrsim \Big| S_{N} \Big(\frac{t_{j}(x)}{2\pi \lambda^{j}}  \Big)  f_j(x)\Big| - | A_1 | - | A_2 | \, ,
\end{equation}
where
\begin{equation}\nonumber
A_1 \! := \! \sum_{1 \leq k < j}  S_{N} \Big( \frac{t_{j}(x)}{2\pi \lambda^{j}} \Big) f_k(x) 
\qquad \text{and}\qquad
A_2 \! := \! \sum_{k > j}  S_{N} \Big( \frac{t_{j}(x)}{2\pi \lambda^{j}} \Big) f_k(x) \, .
\end{equation}
We have already proved that, for $x \in \Gamma$,
$$
\Big|  S_{N} \Big(\frac{t_{j}(x)}{2\pi \lambda^{j}}  \Big)  f_j(x) \Big| \gtrsim \lambda^{j \delta }\quad \text{and}\quad 
| A_2 | \leq \sum_{k > j} \lambda^{- k \delta} \ \lesssim \ 1\, .
$$
On the other hand, by bounding the terms trivially and taking $\lambda$ sufficiently large, we can also arrange that 
$$
| A_1| \leq \sum_{1 \leq k < j} \lambda^{k \delta } \leq \frac{1}{2}\Big| S_{N} \Big(\frac{t_{j}(x)}{2\pi \lambda^{j}}  \Big)  f_j(x)\Big| \, .
$$
Thus, for any $x \in \Gamma $ where the solution is defined, we have
\begin{equation*}
\Big| u \Big(x, \frac{t_{j}(x)}{2\pi \lambda^{j}}\Big) \Big|
= \lim_{N \to \infty}
\Big| S_{N} \Big( \frac{t_{j}(x)}{2\pi \lambda^{j}} \Big)  u_0(x) \Big| \gtrsim \lambda^{j \delta } \, ,
\end{equation*}
so 
there is a sequence of times $\frac{ t_{j}(x)}{2\pi \lambda^{j}}$
for which
$$
\Big| u \Big( x,\frac{t_{j}(x)}{2\pi \lambda^{j}}\Big) \Big| \to \infty \qquad \mbox{as} \qquad \frac{t_{j}(x)}{2\pi \lambda^{j}}  \to 0 \, .
$$

Now, recalling that 
$ s < \frac{(n-1) \sigma}{2}+\frac{1}{4} - \delta$, the proof would be complete if we 
could prove that the $\alpha$--Hausdorff measure of $\Gamma$
were positive, taking $\delta$ and $\sigma$ sufficiently close to $0$ and $\frac{1+2(n-\alpha) }{2(n+1)}$, respectively. Considering the slices $\Gamma_{\!x_1}$, defined via
$$
\Gamma=\left\{x\in\R^n \, : \,  x_1\in(0,1/2),\quad \bar{x}\in\Gamma_{\!x_1}\right\},
$$
it would 
suffice to prove that the~$(\alpha-1)$--Hausdorff measure of $\Gamma_{\!x_1}$ is positive  for all~$x_1\in(0,1/2)$; see 
for instance~\cite[Proposition~7.9]{Falconer2}.
For this we must choose the  modulation directions~$\theta_{j} \in \mathbb{S}^{n-2}$ appropriately, via the ergodic argument of the second section ($\theta_{j} \in (0,1)$ if $n=2$). 
Note that $X_{t\theta_{j}}^{j}$ is a union of disjoint open cubes of side-length $\varepsilon_{2} \lambda^{-j}$, 
while $X^{k, \delta}_{\lambda^{k-j} t\theta_{k}}$ is a union of disjoint closed cubes of side-length $\varepsilon_{2} \lambda^{-(1 - 2\delta)k}$. The 
distance between the cubes is approximately $\lambda^{(\sigma-1)j}$ in the case of the former and $\lambda^{(\sigma-1)k}$ in the case of the latter. 
Thus we see that 
$\Gamma_{\!t\theta_{j}}^{j}$ is a union of disjoint open sets $\QQ(\bar{x}, \varepsilon_{2} \lambda^{-j})$ that we call pseudo-cubes.

\subsubsection*{Case $\alpha=n$}  In this case, the $(n-1)$-dimensional Lebesgue measure $|\cdot|$  of the pseudo-cubes is comparable to actual cubes;
\begin{equation}
\begin{split}\nonumber
 | & \QQ(\bar{x}  , \varepsilon_{2} \lambda^{-j}) | \, \geq \, |Q(\bar{x}, \varepsilon_{2} \lambda^{-j})| - 
   \Big| Q(\bar{x}, \varepsilon_{2} \lambda^{-j}) \cap \!\!\!\! \bigcup_{j<k\le 2j} \!\!\! X^{k, \delta}_{\lambda^{k-j} t\theta_{k}}\Big|
  \\
  & \simeq \,
   \varepsilon_{2}^{n-1} \lambda^{-(n-1)j} - \varepsilon_{2}^{n-1} \lambda^{-(n-1)j} \!\!\! 
   \sum_{k=j+1}^{2j} \!\!\! \lambda^{-(n-1)(1-2\delta)k} \lambda^{(n-1)(1-\sigma)k}
 \gtrsim\ \, \varepsilon_{2}^{n-1} \lambda^{-(n-1)j} \, ,
 \end{split} 
 \end{equation}
 where we have taken $\lambda$ sufficiently large (recalling $\delta < \sigma/4$). 
Thus, using Corollary~\ref{Corollary:ToroImpr} with $d=n-1$, $\gamma = \alpha-1$ and $R=\lambda^{j}$, we can choose the $\theta_{j}$
so that $|\Gamma^j_{\!x_1}|\gtrsim1$ for all~$x_{1}\in (0,1/2)$, provided that~$j$ is sufficiently large and $\sigma < \frac{n}{2(n+1)}$. From this we see that
 $$
 \lim_{j\to\infty} \Big|\bigcup_{k\ge j}\Gamma^k_{\!x_1}\Big|\gtrsim1 \, ,
 $$
 and,
 since this is a decreasing sequence of sets that are contained in a set with finite $(n-1)$-dimensional Lebesgue measure,
 we can conclude that
 $$
 |\Gamma_{\!x_1}|=\Big|\bigcap_{j\ge1}\bigcup_{k\ge j}\Gamma^k_{\!x_1}\Big|\gtrsim 1 \, ,
 $$
 for all $x_1\in (0,1/2)$.
 This completes the proof in the case $\alpha=n$.
 
\subsubsection*{Case $\alpha<n$} 
We will prove that the $\beta$--Hausdorff measure of~$\Gamma_{\!x_1}$ is positive  
for any $\beta$ in the 
interval $( \frac{(n-1)(2\alpha +1)}{2(n+1)}, \alpha-1 )$.
Note that the interval is not empty if we
restrict to~$\alpha > \frac{3n+1}{4}$. This is enough  
to complete the proof, as we could have started with an $\alpha'>\alpha \geq \frac{3n+1}{4}$ that also satisfies 
$$
s < \frac{n}{2(n+1)}+\frac{n-1}{2(n+1)}(n-\alpha') \, ,
$$
and performed all of the previous arguments for this $\alpha'$.

Considering the Hausdorff content of a set $E \subset \mathbb{R}^{d}$ defined by
$$
\HH^{\beta}_{\infty}(E) := \inf \Big\{ \sum_{i} \delta_{i}^{\beta} :  E \subset \bigcup_{i} Q(x_i, \delta_i)\Big\} \, , 
$$
by the triangle inequality as before, we have that   
\begin{equation}\nonumber
\begin{split}
\HH^{\beta}_{\infty}(\QQ(\bar{x}, \varepsilon_{2} \lambda^{-j}) )  \,    &  \geq  \, \HH^{\beta}_{\infty} (Q(\bar{x}, \varepsilon_{2} \lambda^{-j})) - 
  \HH^{\beta}_{\infty} \Big(Q(\bar{x}, \varepsilon_{2} \lambda^{-j}) \cap \!\!\!\! \bigcup_{j<k\le 2k} \!\!\! X^{k, \delta}_{\lambda^{k-j} t\theta_{k}} \Big)
   \\
    & \gtrsim \,
      \varepsilon_{2}^{\beta} \lambda^{-\beta j} \! - \! \varepsilon_{2}^{n-1} \lambda^{-(n-1) j} \!\!\! \sum_{k=j+1}^{2j} \!\!\!  \lambda^{- \beta (1-2\delta) k}  \lambda^{ (n-1)(1-\sigma)k }
     \gtrsim \, \varepsilon_{2}^{\beta} \lambda^{-\beta j} \, ,
 \end{split} 
 \end{equation}
 using
  $ (1-2\delta) \beta -  (n-1)(1 - \sigma)  > 0$ and taking~$\lambda$ sufficiently
  large. This holds, taking $\delta$ and $\sigma$ close enough to $0$ and~$\frac{1+2(n-\alpha) }{2(n+1)}$, 
  respectively, since we have restricted to~$\beta > \frac{(n-1)(2\alpha +1)}{2(n+1)}$.  
Again we see that, in this range of $\beta$, the $\HH^{\beta}_{\infty}$-content of the pseudo-cubes is comparable to that of the 
 actual cubes.
 
 We now use Corollary~\ref{Corollary:ToroImpr}, 
 with $d=n-1$, $\gamma = \alpha-1$ and $R=\lambda^{j}$, 
  to choose~$\theta_j$ such that, for all~$x_{1}\in (0,1/2)$, the
 $\Gamma_{\!x_1}^{j}$ are  unions of pseudo-cubes whose 
 centres are~$\varepsilon_{2} \lambda^{-j\frac{\alpha-1}{n-1}}$-dense
 in~$B(0,1/2) \subset \R^{n-1}$, when~$j$ is sufficiently large. Recalling that as the sidelengths are shorter, of length $\varepsilon_{2} \lambda^{-j}$, this is not enough to come close to covering the ball as before. However, discarding some pseudo-cubes, if necessary, 
 we find that~$\Gamma^{j}_{\!x_1}$ contains 
 a set of pseudo-cubes whose centres are
 a {\it quasi-lattice} with separation~$\lambda^{-j\frac{\alpha-1}{n-1}}$; see~\cite[Lemma 4]{LuR2}.
 That is to say, for any $\bar{y}\in B(0,1/2) \cap \lambda^{-j\frac{\alpha-1}{n-1}} \mathbb{Z}^{n-1}$
 there exists a unique centre~$\bar{x}$ 
 satisfying~$|\bar{x} - \bar{y}| \leq \lambda^{-j\frac{\alpha-1}{n-1}}$.

 Using a density theorem due to Falconer~\cite{Falconer1} (see also \cite[Proposition 8.5]{Falconer2} for a similar theorem), 
the positivity of the $\beta'$--Hausdorff measure of~$\Gamma_{\!x_1}$, for any $\beta' < \beta$, is a 
consequence of the following density property
 \begin{equation}\label{WWNTP}
    \liminf_{j \to \infty} \HH^{\beta}_{\infty} ( \Gamma^j_{\!x_1} \cap Q(\bar{x}, \delta) ) \geq c\delta^{\beta} , 
    \qquad 
    \forall\ Q(\bar{x}, \delta)\subset B(0,1/2),
    \quad
    \forall\ \delta >0 \, .
   \end{equation}
   Thus it would be sufficient for us to show that \eqref{WWNTP} holds for~$\beta \in ( \frac{(n-1)(2\alpha +1)}{2(n+1)}, \alpha-1 )$.   
  Essentially this means that the most efficient way to cover~$\Gamma^j_{\!x_1} \cap Q(\bar{x}, \delta)$ is with 
   a single cube of side~$\delta$. 
   The only real competitor is the cover
   that consists of the disjoint union of cubes of side-length $\varepsilon_{2} \lambda^{-j}$ placed on the 
   top of the pseudo-cubes of the quasi-lattice. However, this cover is costed at
   $$
   \sum_{i} \delta_i^\beta \simeq  \left( \frac{\delta}{ \lambda^{-j\frac{\alpha-1}{n-1}}} \right)^{\!\!n-1} \!\!\! (\varepsilon_{2} \lambda^{-j})^\beta
   =\varepsilon_{2}^{\beta}\delta^{n-1} \lambda^{j(\alpha-1-\beta)} \, ,
   $$
  which diverges as $j\to\infty$ (recalling that $\beta < \alpha-1$). 
  The remaining coverings are ruled out in exactly the same way as in~\cite[Section 4]{LuR2}.
  The only requirement is that the $\HH^{\beta}_{\infty}$-content of the pseudo-cubes is comparable to that 
  of the  actual cubes, which we have already observed, so the proof is complete.
   \hfill $\Box$ 
 %
 %

\end{document}